\documentclass
[
    fontsize = 11 pt,
    american,
    captions = tableheading,
    numbers = noenddot,
    abstracton,
    paper = letter,
]
{scrartcl}

\usepackage[T1]{fontenc}
\usepackage[utf8]{inputenc}
\usepackage
[
    babel = true, % Enables language-specific tuning.
]
{microtype}           % Uses the text space more efficiently.
\usepackage{csquotes} % Uses the correct quotes according to the current language.

\usepackage[dvipsnames]{xcolor} % Allows it to define colors. The option says that common names can be used.

\definecolor{stroke1}{HTML}{2574A9} % This color is used as the standard color to highlight things.

\usepackage{amsmath}
\usepackage{amssymb}
\usepackage{amsthm}
\usepackage{thmtools}
\usepackage{mathtools}
\usepackage{thm-restate}
\usepackage{dsfont}        % Yields far better blackboard-bold letters than \mathbb. Use \mathds in order to write such letters.
\usepackage{braceMnSymbol} % Adjusts overbraces and underbraces such that longer versions are put together seamlessly.

\usepackage
[
    ttscale = 0.85, % Scales the typewriter font.
]
{libertine} % The main font used in this thesis.
\usepackage
[
    libertine,    % Changes the math font to libertine (the main font).
    slantedGreek, % Makes all greek letters italic by default. If you want to use an upright greek letter, use ›\up‹ immediately followed by the letter’s name. For example, \upGamma displays an upright uppercase gamma.
    vvarbb,       % Changes the \mathbb font to another font. However, \mathbb remains ugly and should not be used. Use \mathds instead.
    libaltvw,     % Uses different characters for v und w that look far better than the default ones.
]
{newtxmath} % The main math font of this thesis. It fits well with the main font.
\usepackage{url} % Responsible for URL formatting.
\usepackage{bm}  % Allows to use sensible bold letters in math mode. This package has to go after the font packages. Otherwise it does not work correctly!

\usepackage{graphicx} % The standard package for including graphics into your document.
\usepackage
[
    subrefformat = simple, % Formats the label of the \subref command without parentheses.
    labelformat = simple,  % Formats the mark of a subfigure without parentheses.
]
{subcaption}         % Enables it to have subfigures inside of a single figure.

\usepackage{array}     % Improves the way that tables can be formatted.
\usepackage{booktabs}  % Adds lines (called ›rules‹) that can be used in tables and improves spacing.
\usepackage{longtable} % Allows to make tables that span multiple pages.
\usepackage{pdflscape} % Allows to change a page into landscape. This is handy if a table is very wide.

\usepackage[shortlabels]{enumitem} % Adds tons of useful features to enumeration environments.

\usepackage
[
    ruled,         % Creates lines at the top and at the bottom. Further, the caption is now above the algorithm.
    vlined,        % Shows the scope of a statement spanning multiple lines via a small vertical bar. Thus, no closing tags are needed.
    linesnumbered, % Shows line numbers.
]
{algorithm2e} % Allows to write pseudocode.

\usepackage{xspace}   % Adds the functionality that a space after a command will be shown as a space in the output.
\usepackage
[
    shortcuts, % Allows to use short symbols for non-breaking hyphens and dashes instead of lengthy commands.
]
{extdash}             % Adds non-breaking hyphens and dashes.
\usepackage{setspace} % Allows to easily chnage the spacing inside of the document.

\usepackage{xparse}    % Is used in order to define reasonable commands.
\usepackage{footnote}  % Allows it to extend the environments footnotes can be used in. It is said that this package is in conflict with ›hyperref‹. I did not note any troubles. However, if something is fishy, it is probably best to not use this package.
\usepackage{afterpage} % Adds the \afterpage command, which specifies that the provided argument shall be processed after the current page is finished.
\usepackage
[
    textsize = scriptsize, % Determines the text size of the TODO note.
]
{todonotes}            % Adds TODO notes to the document. These are small text areas inside of the margin of a page.

\usepackage
[
    sortcites,              % Sort multiple references when citing them together.
    style = alphabetic,     % The style of a citation mark.
    defernumbers,           % Makes sure that references always have unique numbers. This is important if you use multiple bibliographies.
    safeinputenc,           % Allows to use UTF8 characters in the bibliography and tries to translate them into TeX automatically.
    backref = true,         % Creates back references in the bibliography.
    backrefstyle = three,   % Compresses three or more consecutive pages in the back references into a range.
    hyperref = true,        % Makes links generated by biblatex clickable. If hyperref is not used, a warning is issued.
    maxbibnames = 99,       % The maximum number of names displayed in the bibliography.
    maxcitenames = 3,       % The maximum number of names displayed when using commands like ›textcite‹. The default is 3. After that, ›et al.‹ is used.
    minalphanames = 3,       % The minimum number of names to display for the alphabetic sorting style.
]
{biblatex} % Used in order to format the bibliography.

\DefineBibliographyStrings{english}%
{%
    backrefpage  = {\lowercase{s}ee page}, % For a single page number.
    backrefpages = {\lowercase{s}ee pages} % For multiple page numbers.
}

\usepackage{multirow}
\usepackage{caption}
\usepackage
[
    bookmarks = true,                 % Generates boodmarks for the PDF.
    bookmarksopen = false,            % The bookmarks are closed by default.
    bookmarksnumbered = true,         % The bookmarks use the numbers of the corresponding headline.
    pdfstartpage = 1,                 % The first page seen when opening the PDF.
    colorlinks = true,                % The text of hyperlinks is colored instead of having a colored box around it.
    allcolors = stroke1,              % Every hyperlink uses the same color. If you want to change specific colors, use the commands below.
]
{hyperref} % The standard package that is used for creating hyperlinks inside of a document.

\usepackage
[
    noabbrev,   % The words in front of the labels are not abbreviated.
    nameinlink, % Extends the link of a reference to the word in front of it.
]
{cleveref} % This package must be included after ›hyperref‹. It creates clever references that know what they refer to.

\date{}

\makeatletter
    \def\IfEmptyTF#1%
    {%
        \if\relax\detokenize{#1}\relax%
            \expandafter\@firstoftwo%
        \else%
            \expandafter\@secondoftwo%
        \fi%
    }
\makeatother

\NewDocumentCommand{\mathOrText}{m}
{%
    \ensuremath{#1}\xspace%
}

\let\originalleft\left
\let\originalright\right
\renewcommand{\left}{\mathopen{}\mathclose\bgroup\originalleft}
\renewcommand{\right}{\aftergroup\egroup\originalright}

\makeatletter
    \DeclareRobustCommand{\bfseries}%
    {%
        \not@math@alphabet\bfseries\mathbf%
        \fontseries\bfdefault\selectfont%
        \boldmath%
    }

\xspaceaddexceptions{]\}}

\allowdisplaybreaks

\crefname{ineq}{inequality}{inequalities}
\creflabelformat{ineq}{#2{\upshape(#1)}#3}

\crefname{term}{term}{terms}
\creflabelformat{term}{#2{\upshape(#1)}#3}

\crefname{cond}{condition}{conditions}
\creflabelformat{cond}{#2{\upshape(#1)}#3}

\crefname{assume}{assumption}{assumptions}
\creflabelformat{assume}{#2{\upshape(#1)}#3}

\deffootnote[1.2 em]{1.2 em}{0 em}{\makebox[1.2 em][l]{\textbf{\thefootnotemark}}}

\let\oldfootnote\footnote

\newlength{\spaceBeforeFootnote} % Denotes the space before the footnote mark in em.
\newlength{\spaceAfterFootnote}  % Denotes the space after the footnote mark in em.

\RenewDocumentCommand{\footnote}{o o o m}%
{%
    \IfNoValueTF{#1}%
    {%
        \oldfootnote{#4}%
    }%
    {%
        \setlength{\spaceBeforeFootnote}{\IfEmptyTF{#1}{0}{#1} em}%
        \IfNoValueTF{#2}%
        {%
            \hspace*{\spaceBeforeFootnote}\oldfootnote{#4}%
        }%
        {%
            \setlength{\spaceAfterFootnote}{\IfEmptyTF{#2}{0}{#2} em}%
            \hspace*{\spaceBeforeFootnote}\IfNoValueTF{#3}{\oldfootnote{#4}}{\oldfootnote[#3]{#4}}\hspace*{\spaceAfterFootnote}%
        }%
    }%
}

\makesavenoteenv{figure}
\makesavenoteenv{table}
\makesavenoteenv{tabular}

\declaretheoremstyle
[
   	spaceabove = \topsep,
   	spacebelow = \topsep,
   	headfont = \bfseries,
   	headformat = \textcolor{stroke1}{$\blacktriangleright$} \NAME~\NUMBER \NOTE,
   	notefont = \bfseries,
   	notebraces = {(}{)},
   	bodyfont = \normalfont,
   	postheadspace = 0.5 em,
   	qed = \textcolor{stroke1}{\bfseries$\blacktriangleleft$},
]
{myTheoremStyle}

\declaretheorem
[
   	style = myTheoremStyle,
   	name = Lemma,
    sharenumber = conjecture,
]
{lemma}
\declaretheorem
[
   	style = myTheoremStyle,
   	name = Corollary,
    sharenumber = conjecture,
]
{corollary}
\declaretheorem
[
   	style = myTheoremStyle,
   	name = Theorem,
    sharenumber = conjecture,
]
{theorem}

\NewDocumentCommand{\functionTemplate}{m m m m o}%
{%
    \IfNoValueTF{#5}%
    {%
        \mathOrText{#1\left#2{#4}\right#3}%
    }%
    {%
        \mathOrText{#1#5#2{#4}#5#3}%
    }%
}

\newcommand*{\leftBracketType}{(}
\newcommand*{\rightBracketType}{)}

\NewDocumentCommand{\createFunction}{m m o o}%
{%
    \renewcommand*{\leftBracketType}{\IfNoValueTF{#3}{(}{#3}}%
    \renewcommand*{\rightBracketType}{\IfNoValueTF{#4}{)}{#4}}%
    \NewDocumentCommand{#1}{o o}%
    {%
        \IfNoValueTF{##1}%
        {%
            \mathOrText{#2}%
        }%
        {%
            \functionTemplate{#2}{\leftBracketType}{\rightBracketType}{##1}[##2]%
        }%
    }%
}

\DeclareDocumentCommand{\probabilisticFunctionTemplate}{m m O{} o}
{%
    \functionTemplate{#1}%
    {\lbrack}%
    {\rbrack}%
    {#2\IfEmptyTF{#3}{}{\ \IfNoValueTF{#4}{\left}{#4}\vert\ \vphantom{#2}#3\IfNoValueTF{#4}{\right.}{}}}%
    [#4]%
}

\newcommand*{\N}{\mathOrText{\mathds{N}}}

\newcommand*{\R}{\mathOrText{\mathds{R}}}

\newcommand*{\indicatorFunctionSymbol}{\mathds{1}}

\RenewDocumentCommand{\Pr}{m O{} o}%
{%
    \probabilisticFunctionTemplate{\mathrm{Pr}}{#1}[#2][#3]%
}

\NewDocumentCommand{\E}{m O{} o}%
{%
    \probabilisticFunctionTemplate{\mathds{E}}{#1}[#2][#3]%
}

\NewDocumentCommand{\Var}{m O{} o}%
{%
    \probabilisticFunctionTemplate{\mathrm{Var}}{#1}[#2][#3]%
}

\DeclareDocumentCommand{\bigO}{m o}%
{%
    \functionTemplate{\mathrm{O}}{(}{)}{#1}[#2]%
}

\DeclareDocumentCommand{\smallO}{m o}%
{%
    \functionTemplate{\mathrm{o}}{(}{)}{#1}[#2]%
}

\DeclareDocumentCommand{\bigTheta}{m o}%
{%
    \functionTemplate{\upTheta}{(}{)}{#1}[#2]%
}

\DeclareDocumentCommand{\bigOmega}{m o}%
{%
    \functionTemplate{\upOmega}{(}{)}{#1}[#2]%
}

\DeclareDocumentCommand{\smallOmega}{m o}%
{%
    \functionTemplate{\upomega}{(}{)}{#1}[#2]%
}

\DeclareDocumentCommand{\eulerE}{o}%
{%
    \mathOrText{\mathrm{e}\IfNoValueTF{#1}{}{^{#1}}}%
}

\DeclareDocumentCommand{\poly}{m o}%
{%
    \functionTemplate{\mathrm{poly}}{(}{)}{#1}[#2]%
}

\createFunction{\id}{\mathrm{id}}

\NewDocumentCommand{\ind}{m o o}%
{%
    \IfNoValueTF{#2}%
    {%
        \mathOrText{\indicatorFunctionSymbol_{#1}}%
    }%
    {%
        \functionTemplate{\indicatorFunctionSymbol_{#1}}{(}{)}{#2}[#3]%
    }%
}

\DeclareDocumentCommand{\dom}{m o}%
{%
    \functionTemplate{\mathrm{dom}}{(}{)}{#1}[#2]%
}

\DeclareDocumentCommand{\rng}{m o}%
{%
    \functionTemplate{\mathrm{rng}}{(}{)}{#1}[#2]%
}

\DeclareDocumentCommand{\d}{o}%
{%
    \mathrm{d}\IfNoValueTF{#1}{}{^{#1}}%
}

\DeclareDocumentCommand{\set}{m m o}%
{%
    \mathOrText{\IfNoValueTF{#3}{\left}{#3}\{#1\ \IfNoValueTF{#3}{\left}{#3}\vert\ \vphantom{#1}#2\IfNoValueTF{#3}{\right.}{}\IfNoValueTF{#3}{\right}{#3}\}}%
}

\DeclareDocumentCommand{\randomProcess}{m o}
{%
    \mathOrText{X^{(#1)}\IfNoValueTF{#2}{}{_{#2}}}%
}

\DeclareDocumentCommand{\transformedProcess}{o}
{%
    \mathOrText{Y\IfNoValueTF{#1}{}{_{#1}}}%
}

\DeclareDocumentCommand{\filtration}{o}
{%
    \mathOrText{\mathcal{F}\IfNoValueTF{#1}{}{_{#1}}}%
}

\newcommand*{\numberOfVertices}{\mathOrText{n}}

\newcommand*{\infectionRate}{\mathOrText{\lambda}}

\newcommand*{\infectionExponent}{\mathOrText{\alpha}}
\newcommand*{\infectionThreshold}{\mathOrText{\beta}}

\newcommand*{\contactProcess}{\mathOrText{C}}

\newcommand*{\infectedDiscrete}[1]{\mathOrText{I_{\timeContinuous{#1}}}}

\newcommand*{\timeDiscrete}{\mathOrText{t}}
\newcommand*{\timeContinuous}[1]{\mathOrText{\tau_{#1}}}

\DeclareDocumentCommand{\lyapunovHelper}{o}
{%
    \mathOrText{f\IfNoValueTF{#1}{}{(#1)}}%
}

\DeclareDocumentCommand{\lyapunovFunction}{m o}
{%
    \mathOrText{F\IfNoValueTF{#2}{}{(#2)}}%
}

\DeclareDocumentCommand{\potentialFunctionIMinusR}{o}
{%
    \mathOrText{H\IfNoValueTF{#1}{}{(#1)}}%
}

\title{From Market Saturation to Social Reinforcement: Understanding the Impact of Non-Linearity in Information Diffusion Models}

\author{%
    Tobias Friedrich$^{*}$ \and Andreas Göbel$^{*}$ \and Nicolas Klodt$^{*}$ \and Martin~S. Krejca$^{\dagger}$ \and Marcus Pappik$^{*}$
}

\publishers
 {
  	\footnotesize $^{*}$ Hasso Plattner Institute, University of Potsdam, Potsdam, Germany\\
  	\{tobias.friedrich, andreas.goebel, nicolas.klodt, marcus.pappik\}@hpi.de \\[1 ex]

  	$^{\dagger}$ LIX, CNRS, Ecole Polytechnique, Institut Polytechnique de Paris, Palaiseau, France \\
  	martin.krejca@polytechnique.edu
 }

\begin{document}

\maketitle

\begin{abstract}
Diffusion of information in networks is at the core of many problems in AI. Common examples include the spread of ideas and rumors as well as marketing campaigns. Typically, information diffuses at a non-linear rate, for example, if markets become saturated or if users of social networks reinforce each other's opinions. Despite these characteristics, this area has seen little research, compared to the vast amount of results for linear models, which exhibit less complex dynamics. Especially, when considering the possibility of re-infection, no fully rigorous guarantees exist so far.

We address this shortcoming by studying a very general non-linear diffusion model that captures saturation as well as reinforcement. More precisely, we consider a variant of the  SIS model in which vertices get infected at a rate that scales polynomially in the number of their infected neighbors, weighted by an infection coefficient $\infectionRate$. We give the first fully rigorous results for thresholds of $\infectionRate$ at which the expected survival time becomes super-polynomial. For cliques we show that when the infection rate scales sub-linearly, the threshold only shifts by a poly-logarithmic factor, compared to the standard SIS model. In contrast, super-linear scaling changes the process considerably and shifts the threshold by a polynomial term. For stars, sub-linear and super-linear scaling behave similar and both shift the threshold by a polynomial factor.
Our bounds are almost tight, as they are only apart by at most a poly-logarithmic factor from the lower thresholds, at which the expected survival time is logarithmic.
\end{abstract}

\newpage

\section{Introduction}
\label{sec:introduction}

Information diffusion processes on graphs are widely studied in the area of AI \cite{kong2023online,SunCM2023SocialInfluenceMaximization,JiangRF2023PoliticalLearning,RazaqueRKAR2022Survey,yan2019maximization,wilder2018end,ikeda2016examination} and in other domains, modeling various graph processes, such as spread of infections \cite{Pastor-SatorrasCVMV15Survey,leskovec2007cost} and computer viruses \cite{berger2005spread,BorgsCGS10Antidote}, social influence and the spread of ideas~\cite{kempe2003maximizing}, and viral marketing campaigns \cite{agarwal2008blogosphere}.

Commonly, information diffusion processes are modeled as epidemiological models over networks (see \cite{Pastor-SatorrasCVMV15Survey} for an extensive survey). In these models, each vertex of the host network is in a state, such as \emph{susceptible} or \emph{infected}, and transitions between these states at variable rates that depend on the states of all vertices in the network. A very prominently studied epidemiological model is the \emph{SIS model}. In this model, each susceptible vertex becomes infected by each of its infected neighbors independently with a system-wide infection rate of $\infectionRate \in \R_{> 0}$, and each infected vertex turns susceptible independently with a normalized rate of~$1$.

The SIS-model includes the possibility for vertices to re-infect after recovering from the infection. That accounts, for example, for infections that do not grant immunity~\cite{newman2003structure} or for bloggers that can post the same message multiple times~\cite{leskovec2007patterns}. Therefore, it is possible for the infection to stay active for a very long time by infecting the same vertices over and over again. The quantity that measures how long the process contains infected vertices is known as the \emph{survival time} and marks an important property for networks in the SIS model.

Due to its relevance, the survival time of the standard SIS model has been studied extensively for decades both empirically~\cite{Pastor-SatorrasCVMV15Survey} as well as mathematically rigorously, for the latter on infinite~\cite[e.g.,][]{Harris74,Liggett96InfiniteBinaryTrees,NamNS22SISinfinite} and on finite graphs~\cite[e.g.,][]{berger2005spread,ganesh2005effect,BorgsCGS10Antidote}.
Combined, these results show for a large variety of different finite graph classes a sharp transition, with respect to the infection rate~\infectionRate, from a survival time logarithmic in the graph size to one that is super-polynomial.
The regime of~\infectionRate where this change occurs is known as the \emph{epidemic threshold}. This regime is mostly independent of the starting configuration of the process as long as at least one vertex starts infected.

In the majority of the epidemiology models studied, vertices get infected at a rate that scales linearly with the number of their infected neighbors. However, experiments have shown that there are processes that cannot be modeled with such simple assumptions~\cite{hodas2014simple,monsted2017evidence}. For example, when behavior is spread over social networks, there is a social reinforcement effect that leads to much higher adoption rates when the number of infected neighbors is high~\cite{Centola2010SocialReinforcement}. Similar effects can be observed for biological contagions~\cite{StOnge2021UniversalNI}. In some cases the opposite effect happens, for example, during market saturation, when trying to convince customers to buy a product~\cite{kong2023online}.

As non-linear infection rates are prominent in natural processes, it is important to study how they alter the insights that have been gained for linear processes.
To this end, we consider an altered version of the standard SIS model, parameterized by an \emph{infection exponent} $\infectionExponent \in \R_{> -1}$.
In this version, susceptible vertices do not get infected by their neighbors independently but, instead, a susceptible vertex with~$i$ infected neighbors is infected with a rate of $\infectionRate i^{1 + \infectionExponent}$. Note that we call $\infectionRate$ the \emph{infection coefficient}, and for $\infectionExponent=0$, it coincides with the infection rate of the standard SIS model.

This model as well as generalizations and variants thereof have already been theoretically studied~\cite{StOnge2021NonlinearInfectionsHypergraphs,zhang2015health,hethcote1991some,liu1986influence}. These works use mean-field theory, simplifying the original process by making approximation assumptions. In these works the host graph is assumed to be a clique in order to derive differential equations that model the dynamics of the process. This results in a threshold for the infection coefficient at which the simplified process changes from having the all-susceptible state as unique global equilibrium to having two equilibria from which one is not the all-susceptible state and is globally stable. This threshold gives an estimate for the epidemic threshold on cliques. To the best of our knowledge, no fully rigorous results on the epidemic threshold exist so far.

\paragraph{Our Contribution}

\begin{table}[]
    \centering
    \caption{Our threshold results for the infection coefficient $\infectionRate$ in the modified SIS process with infection rate $\infectionRate I^{1+\infectionExponent}$, where $I$ is the number of infected neighbors of a vertex. The table gives regimes for~$\infectionRate$ depending on the host graph structure, the number of vertices $\numberOfVertices$ and the infection exponent $\infectionExponent$, in which the expected survival time is logarithmic or super-polynomial in $\numberOfVertices$ respectively. The case $\infectionExponent =0$ corresponds to the known results for the standard SIS model.}
    \label{table:results}
    \begin{tabular}{l l l}
        \toprule
        & $T\in \bigO{\log(\numberOfVertices)}$ & $T\geq \numberOfVertices^{\smallOmega{1}}$\\
        \midrule
        \textbf{Clique} & &\\
         \, $\infectionExponent<0$&$\infectionRate \in \bigO{\numberOfVertices^{-1}}$ & $\infectionRate \in \smallOmega{\numberOfVertices^{-1}\log( n)^{-\infectionExponent}}$\\
         \, $\infectionExponent>0$ & $\infectionRate \in \smallO{\numberOfVertices^{-1-\infectionExponent}}$ & $\infectionRate \in\smallOmega{\numberOfVertices^{-1-\infectionExponent} \log( n)^{\infectionExponent}}$ \\
         \, $\infectionExponent=0$ & $\infectionRate\leq \numberOfVertices^{-1}$ \cite{ganesh2005effect} & $\infectionRate\geq (1+\varepsilon)\numberOfVertices^{-1}$ \cite{ganesh2005effect} \footnote{Note that \cite{ganesh2005effect} states a slightly better bound. However, the bound follows from Corollary 4.1 which needs parameter $r$ to be smaller than 1 uniformly in $n$. That is not fulfilled for the given clique bound. We adjusted the bound for this condition to be fulfilled.} \\
         \textbf{Star} & \\  & $\infectionRate\in\bigO{\numberOfVertices^{-\frac{1}{2}-\frac{\infectionExponent}{2(2+\infectionExponent)}}}$ & $\infectionRate\in\smallOmega{\numberOfVertices^{-\frac{1}{2}-\frac{\infectionExponent}{2(2+\infectionExponent)}}\log (\numberOfVertices)^{\frac{4}{1+\infectionExponent}}}$\\
        \bottomrule
    \end{tabular}
\end{table}

We analyze the epidemic threshold of a SIS variant with non-linear infection rates in a rigorous mathematical manner.
To the best of our knowledge, these are the first results that study the process in the non-linear setting without any simplifications.
We prove both upper and lower bounds for the epidemic threshold for cliques (\Cref{cor:cliqueSubLinear,cor:cliqueSuperLinear}, assuming one initially infected vertex) and for stars (\Cref{cor:star}, assuming the center to be initially infected).
Further, our results encompass settings with sub-linear as well as super-linear infection rates in case of a constant infection exponent~\infectionExponent.
In all cases, our upper and lower bounds of the epidemic threshold are almost tight, i.e. different by at most poly-logarithmic factors. Our results are summarized in \Cref{table:results}. Note that in this setting the survival time increases monotonically when adding extra vertices and edges, hence our lower bounds carry over to graphs with large clique or star subgraphs.

For cliques of size~$n$, we see a clear difference between the sub-linear and the super-linear setting.
For sub-linear infection rates (\Cref{cor:cliqueSubLinear}, $\infectionExponent \in (-1, 0)$), the epidemic threshold remains similar to the linear setting~\cite[Section~V.C]{ganesh2005effect} and increases from $1/n$ to $\smallOmega{\log (n)^{-\infectionExponent}/n}$.

For the super-linear setting for cliques of size~$n$ (see \Cref{cor:cliqueSuperLinear}), the infection exponent~\infectionExponent of the model also has an impact on the power of~$n$.
The infection survives already, in expectation, a time exponential in~$n$ once $\infectionRate \in \smallOmega{n^{-1 - \infectionExponent} \log (n)^{\infectionExponent}}$, decreasing the threshold from the linear setting by a factor in the order of $\log(n) / n^{\infectionExponent}$. However, the survival time has a very high variance, in the sense that the process survives exponentially long with a probability at most super-polynomial in $-n$ (see \Cref{lem:star_positive_p}). This long survival time occurs once the process hits a critical mass of infected vertices. In contrast, the process in the linear model survives exponentially long, already with a probability linear in $n^{-1}$.

For stars with~$n$ leaves, we provide a unified bound for the sub- and super-linear setting (see \Cref{cor:star}).
In both cases, the epidemic threshold deviates roughly by a factor of $n^{-\infectionExponent / (2 (2 + \infectionExponent))}$ from the linear threshold of $n^{-1/2}$~\cite[Theorems~$5.1$ and~$5.2$]{ganesh2005effect}.
This shows that the effect of changing the infection rate is also well pronounced when deviating from the linear setting.

Overall, on stars, changing the infection rate has a strong impact on the epidemic threshold in all settings.
On cliques, the impact is strong when the infection rate scales super-linearly, but the effect is far less prominent in a sub-linear scaling.

\section{Preliminaries}
\label{sec:preliminaries}
We consider a variation of the SIS model in which the rate at which vertices get infected scales non-linearly in the number of its infected neighbors. The process is defined as follows.

Let $G=(V,E)$ be a finite, undirected graph with vertex set $V$ and edge set $E$. Further let $\infectionRate \in \R_{>0}$ and $\infectionExponent \in \R_{>-1}$. A contact process $C$ with infection coefficient $\infectionRate$ and infection exponent $\infectionExponent$ on $G$ is a continuous-time Markov process over partitions of $V$ into susceptible vertices and infected vertices. The transition of the process is decided by Poisson point processes on the vertices, which we call Poisson clocks. Each infected vertex has a Poisson clock with rate $1$ which when it triggers \emph{heals} it and moves it to the susceptible set. Each susceptible vertex has a Poisson clock with a variable rate. For every vertex $v\in V$ and time $\timeDiscrete \in \R_{\geq 0}$, let $N_{\timeDiscrete,v}$ be the number of infected neighbors of $v$. Each susceptible vertex $v$ has a Poisson clock that infects it at rate $\infectionRate N_{\timeDiscrete,v}^{1+\infectionExponent}$. Note that we restrict $\infectionExponent$ to be larger than $-1$ as otherwise the clock rate is either not defined or strictly positive for susceptible vertices with no infected neighbors. We also assume $\infectionExponent$ to be constant in the number of vertices.

We aim to calculate the \emph{survival time} of the contact process, the first point in time at which the process reaches the only absorbing state, which is the the state in which all vertices are susceptible.

In our proofs we sometimes consider the discrete version of that process in which one step is exactly one trigger of a clock. That means that in one step exactly one vertex heals or one vertex gets infected. We use the fact that while at least one vertex is infected all clocks together trigger at a rate of at least 1 and at most polynomial to transfer the bounds from the discrete version to the continuous process.

We use stochastic domination to transfer results from one random variable to another. We say that a random variable $(X_t)_{t \in \R}$ \emph{dominates} another random variable $(Y_t)_{t \in \R}$ if and only if there exists a coupling $(X'_t, Y'_t)_{t \in \R}$ in a way such that for all $t \in \R_{\geq 0}$ we have $X'_t \geq Y'_t$. For example for two contact processes $C$ and $C'$ that only differ in the fact that $C$ has a higher infection coefficient, the number of infected vertices in $C$ dominates the number of infected vertices in $C'$ as the processes can be coupled in a way such that all healing clocks trigger at the same time and infections in $C'$ imply infections in $C$ at the same time. The domination then directly implies that the survival time of $C$ dominates the survival time of $C'$, so the survival time in our model increases monotonically with the infection coefficient.

When we say that some event happens asymptotically almost surely (a.a.s.) that means that for increasing number of vertices in the considered graph the event happens with a probability of $1- \smallO{1}$.

Some of the processes that we analyze are very similar to the well-known gambler's ruin problem, as they increase and decrease by one with certain probabilities until they reach a limit in either direction. We consider the following version of the gambler's ruin problem.

\begin{theorem}[Gambler's ruin~{\cite[page~$345$]{feller1957introduction}}]\label{pre:gamblersRuin}
Let $(P_t)_{t \in \N}$ be the amount of money that a player has in a gambler's ruin game that has a probability of $p \neq 1/2$ for them to win in each step. Let $q=1-p$. The game ends at time $T$ when the player either reaches the lower bound $l$ or the upper bound $u$ of money. Then
\begin{enumerate}
\item $\Pr{P_T = l} = \frac{1-(p/q)^{u-P_0}}{1-(p/q)^{u-l}}$;
\item $\Pr{P_T = u} = \frac{1-(q/p)^{P_0-l}}{1-(q/p)^{u-l}}$.\qedhere
\end{enumerate}
\end{theorem}

Wald's equation helps us calculate the survival time by partitioning the process into phases and then bounding the number and length of those phases.

\begin{theorem}[Wald's equation~{\cite[page~$346$]{mitzenmacher2017probability}}]\label{pre:waldsEquation}
Let $X_1,X_2,...$ be nonnegative, independent, identically distributed random variables with distribution $X$. Let $T$ be a stopping time for this sequence. If $T$ and $X$ have bounded expectation, then

\begin{align*}
    \E{\sum_{i=1}^{T}{X_i}}= \E{T}\cdot \E{X}.
\end{align*}
\end{theorem}

The following theorem bounds the expected value of the maximum of $n$ exponentially distributed random variables which helps as bounding the time until all vertices heal at least once.

\begin{theorem}[{\cite[page~$33$]{mitzenmacher2017probability}}]\label{pre:maxExponential}
Let $n \in \N_{> 0}$, and let $\{X_i\}_{i \in [n]}$ be independent random variables that are each exponentially distributed with parameter $\lambda \in \R_{> 0}$. Let $m = \max_{i \in [n]} X_i$, and let $H_n$ be the $n$-th harmonic number. Then
\begin{align*}
\E{m} &= \frac{H_n}{\lambda} < \frac{1 + \ln(n+1)}{\lambda}.\qedhere
\end{align*}
\end{theorem}

We use Chernoff bounds to bound the value of binomially distributed random variables.

\begin{theorem}[{\cite[Theorem~$4.4$, Theorem~$4.5$]{mitzenmacher2017probability}}]\label{pre:chernoff}
Let $X_1,...,X_n$ be independent Poisson trials, $X= \sum_{i=1}^n{X_i}$, $\mu = \E{X}$ and $\delta \in (0,1)$. Then
\begin{enumerate}
    \item $\Pr{X \geq (1+\delta)\mu}\leq e^{-\mu\delta^2/3}$,

    \item $\Pr{X \leq (1-\delta)\mu}\leq e^{-\mu\delta^2/2}$.
\end{enumerate}
\end{theorem}

\section{Clique}
\label{sec:clique}
In the standard SIS process on a clique there exists a number of infected vertices at which vertices heal and get infected at the same rate, called the equilibrium point. For the contact process with nonlinear infection rate, depending on whether the scaling is sub- or super-linear, this equilibrium is attracting or repelling, respectively. A sub-linear scaling leads to an attracting equilibrium, which yields a threshold close to $1/\numberOfVertices$ (see \Cref{cor:cliqueSubLinear}). A super-linear scaling leads to a repelling equilibrium which makes it very unlikely to reach it. Hence, the infection always dies out fast a.a.s. as shown in \Cref{lem:star_positive_p}. However, there still is a threshold above which the expected survival time is super-polynomial as if the infection crosses the equilibrium, the survival time becomes extremely large. The threshold is roughly $\numberOfVertices^{-1-\infectionExponent}$ (see \Cref{cor:cliqueSuperLinear}) which is different to the sub-linear case.

\subsection{Sub-linear scaling}

When the infection rate scales sub-linearly, there is an equilibrium that is attracting. Also, when being in a state with a number of infected vertices that is a constant factor away from the equilibrium, it is already twice as likely to go towards the equilibrium than going away from it in each step. That means that the survival time is exponential in the equilibrium value. Hence, we get a threshold at the point where this exponential becomes super-polynomial in $\numberOfVertices$. We first show the lower bound on the expected survival time.

\begin{theorem}\label{thm:cliqueSubLinExp}
Let $G$ be a clique with $\numberOfVertices \in \N_{>0}$ vertices. Further, let $\contactProcess$ be a contact process with infection coefficient $\infectionRate \in \R_{>0}$ and infection exponent $\infectionExponent \in (-1,0)$ on $G$ that starts with exactly one infected vertex. Let $T$ be the survival time of $\contactProcess$. If $(\infectionRate \numberOfVertices)^{-1/\infectionExponent} \leq \numberOfVertices/2$ and $(\infectionRate \numberOfVertices/4)^{-1/\infectionExponent} \geq 2$, then $\E{T} \in \bigOmega{2^{(\infectionRate \numberOfVertices)^{-1/\infectionExponent}}/\numberOfVertices}$.
\end{theorem}

\begin{proof}
We show that while there are at most $(\infectionRate\numberOfVertices/4)^{-1/\infectionExponent}$ infected vertices, the probability of infecting a new vertex is at least twice as high as the probability to heal one in the next step. Therefore the process dominates a gambler's ruin instance with a biased coin of probability $2/3$, which has an expected exponential time to reach its lower bound.

Let $c \in \R_{\geq 1}$. Consider a state with $I=(\infectionRate\numberOfVertices/c)^{-1/\infectionExponent}$ infected vertices. From this state, vertices heal at a rate of $I$ and because $(\infectionRate \numberOfVertices)^{-1/\infectionExponent} \leq \numberOfVertices/2$, new vertices get infected at a rate of

\begin{align*}
    \infectionRate I^{1+\infectionExponent}(\numberOfVertices-I) &\geq \infectionRate I^{1+\infectionExponent}\numberOfVertices/2\\
    &= \infectionRate I(\infectionRate\numberOfVertices/c)^{-\infectionExponent/\infectionExponent}\numberOfVertices/2\\
    &=\frac{c}{2}I.
\end{align*}

Note that for $c\geq 4$, the rate at which new vertices get infected is at least twice as high as the rate at which vertices heal. Hence, while there are at most $I=(\infectionRate\numberOfVertices/4)^{-1/\infectionExponent}$ infected vertices, the discrete version of the contact process dominates a gambler's ruin instance $A$ with a biased coin with probability $2/3$.

For all $i\in \N, i<I$, let $p_i$ be the probability that starting with $i$ infected vertices, the infection dies out before reaching $I$ infected vertices. We get using \Cref{pre:gamblersRuin}
\begin{align*}
    p_1&\leq \frac{1-2^{I-1}}{1-2^{I}}\\
    &= \frac{2^{I-1}-1}{2^{I}-1}\\
    &\leq \frac{1}{2}
\end{align*}

and

\begin{align*}
    p_{I-1}&\leq \frac{1-2^{1}}{1-2^{I}}\\
    &= \frac{1}{2^{I}-1}.
\end{align*}

Starting with one infected vertex, the probability to reach $I$ infected vertices before the infection dies out is $1-p_1 \geq 1/2$. From that point on, the number of times that the infection reaches $I$ infected vertices again from below dominates a geometric random variable $X$ with success probability $p_{I-1}\leq \frac{1}{2^{I}-1}$. Note that between each of those times, a vertex has to heal, which happens at a rate of at most $\numberOfVertices$. That means that the expected time between two of those events is at least $1/\numberOfVertices$. Together with Wald's Equation (\Cref{pre:waldsEquation}), that gives us $\E{T} \geq 1/2 \cdot \E{X}/\numberOfVertices$. Plugging in the expected value of $X$ concludes the proof.
\end{proof}

We get a very similar result for the upper bound by upper bounding the probability to increase the number of infected vertices in the next step instead of lower bounding it.

\begin{restatable}{theorem}{cliqueUpper}
Let $G$ be a clique with $\numberOfVertices \in \N_{>0}$ vertices. Further, let $\contactProcess$ be a contact process with infection coefficient $\infectionRate \in \R_{>0}$ and infection exponent $\infectionExponent \in (-1,0)$ on $G$ that starts with exactly one infected vertex. Let $T$ be the survival time of $\contactProcess$. If $\infectionRate\numberOfVertices\geq2$, then $\E{T} \in \bigO{(\infectionRate \numberOfVertices)^{(\infectionRate \numberOfVertices)^{-1/\infectionExponent}}}$.
\end{restatable}

\begin{proof}
    We show that while there are at least $(2\infectionRate\numberOfVertices)^{-1/\infectionExponent}$ infected vertices, the probability of infecting a new vertex is at most half as big as the probability to heal one in the next step. Therefore, the number of infected vertices quickly drops down to that point again when it goes above it. We argue that in any state, the probability of infecting a new vertex is at most $\infectionRate \numberOfVertices$ as big as the probability to heal one in the next step. Therefore the process dominates a gambler's ruin instance with a biased coin of probability $\frac{\infectionRate \numberOfVertices}{\infectionRate \numberOfVertices +1}$, for which we upper-bound the time until it reaches its lower bound. Note that we need $\infectionRate\numberOfVertices \geq2$ because if this value gets too small, the interval we consider is too small to let a gambler's ruin run on it.

    Let $c \in \R_{\geq 1}$. Consider a state with $I=(c\infectionRate\numberOfVertices)^{-1/\infectionExponent}$ infected vertices. From this state, vertices heal at a rate of $I$ and new vertices get infected at a rate of

\begin{align*}
    \infectionRate I^{1+\infectionExponent}(\numberOfVertices-I) &\leq \infectionRate I^{1+\infectionExponent}\numberOfVertices\\
    &= \infectionRate I(c\infectionRate\numberOfVertices)^{-\infectionExponent/\infectionExponent}\numberOfVertices\\
    &=I/c.
\end{align*}

Note that for $c\geq 2$, the rate at which new vertices get infected is at most half as high as the rate at which vertices heal. Hence, while there are at least $I=(2\infectionRate\numberOfVertices)^{-1/\infectionExponent}$ infected vertices, the discrete version of the contact process is dominated by a gambler's ruin Instance $A$ with a biased coin with probability $1/3$. That gives is that the expected time it takes to drop from $I$ infected vertices to $I-1$ infected vertices is at most constant in expectation.

As long as the infection has at least 1 infected vertex, $c$ is upper bounded by $\infectionRate \numberOfVertices$, which also means that in any state, the probability of infecting a new vertex is at most $\infectionRate \numberOfVertices$ as big as the probability to heal one in the next step. Hence, while there are at most $I=(2\infectionRate\numberOfVertices)^{-1/\infectionExponent}$ infected vertices, the discrete version of the contact process is dominated by a gambler's ruin Instance $A$ with a biased coin with probability $\frac{\infectionRate\numberOfVertices}{\infectionRate\numberOfVertices+1}$.

Let $p$ be the probability that starting with $I-1$ infected vertices, the infection dies out before reaching $I$ infected vertices. We get using \Cref{pre:gamblersRuin}
\begin{align*}
    p&\geq \frac{1-(\infectionRate\numberOfVertices)^{1}}{1-(\infectionRate\numberOfVertices)^{I}}\\
    &= \frac{\infectionRate\numberOfVertices-1}{(\infectionRate\numberOfVertices)^{I}-1}\\
    &\geq \frac{1}{(\infectionRate\numberOfVertices)^{I}}\\
\end{align*}

Consider the following variation of the original process. While below $I$ infected vertices, we decrease the rate at which vertices heal such that it is always $\infectionRate\numberOfVertices$ times as likely to infect a new vertex in the next step than healing a vertex. This adjustment only increases the survival time $T$. We now split the adjusted process into phases that start and end when the number of infected vertices increases from $I-1$ to $I$. We showed that the expected time such a phase takes above $I$ infected vertices is constant and below $I$ infected vertices it also takes at most a constant amount of time as it is just a random walk with a biased coin of probability $\frac{\infectionRate\numberOfVertices}{\infectionRate\numberOfVertices+1}\geq \frac{2}{3}$. The number of those phases is dominated by a geometric random variable $X$ with parameter $p$ as each phase independently has a probability of $p$ to be the last one before the infection dies out. That gives us that $\E{T} \leq \E{X}\cdot \bigO{1}$. Plugging in the expected value of $X$ concludes the proof.
\end{proof}

Using those two results, we can pinpoint the threshold relatively precisely.

\begin{corollary}
    \label{cor:cliqueSubLinear}
Let $G$ be a clique with $\numberOfVertices \in \N_{>0}$ vertices. Further, let $\contactProcess$ be a contact process with infection coefficient $\infectionRate \in \R_{>0}$ and infection exponent $\infectionExponent \in (-1,0)$ on $G$ that starts with exactly one infected vertex. Let $T$ be the survival time of $\contactProcess$.

\begin{enumerate}
    \item If $\infectionRate \in \smallOmega{\log (\numberOfVertices)^{-\infectionExponent}/\numberOfVertices}$ then $\E{T}$ is super-polynomial in $\numberOfVertices$.
    \item If $\infectionRate \in \bigO{1/\numberOfVertices}$ then $\E{T}$ is constant in $\numberOfVertices$.\qedhere
\end{enumerate}
\end{corollary}

\subsection{Super-linear scaling}

When the infection rate scales super-linearly, the equilibrium point becomes repellent. When being a constant factor away from that equilibrium, there is a constant drift away from the equilibrium. Therefore, the process is very unlikely to reach this equilibrium, but when it does, the infection might infect all vertices and stay above the equilibrium for a very long time. The expected survival time is therefore decided by the product of the probability of reaching the equilibrium and the expected time the process stays above the equilibrium.

We bound the probability of reaching the equilibrium first. Then we calculate the expected survival time after reaching the equilibrium. Putting those two results together gives us the lower and upper bound for the expected survival time. Note that for the upper bound on the survival time we do not need the second lemma as in the regime we consider the equilibrium value is above the number of vertices and can therefore never be reached.

The following lemma bounds the probability to reach the equilibrium from both sides.

\begin{lemma}\label{lem:star_positive_p}
Let $G$ be a clique with $\numberOfVertices \in \N_{>0}$ vertices. Further, let $\contactProcess$ be a contact process with infection coefficient $\infectionRate \in \R_{>0}$ and infection exponent $\infectionExponent \in \R_{>0}$ on $G$ that starts with exactly one infected vertex. Let $p$ be the probability that the infection reaches a state with $(\infectionRate \numberOfVertices/2)^{-1/\infectionExponent}$ infected vertices. If $1 \leq (\infectionRate \numberOfVertices/2)^{-1/\infectionExponent} \leq n/2$, then

\begin{align*}
    2^{1-(2\infectionRate \numberOfVertices)^{-1/\infectionExponent}} \geq p \geq (\infectionRate \numberOfVertices/2)^{(\infectionRate \numberOfVertices/2)^{-1/\infectionExponent}}.\qedhere
\end{align*}
\end{lemma}

\begin{proof}
    We bound the process below $(\infectionRate \numberOfVertices/2)^{-1/\infectionExponent}$ infected vertices with one gambler's ruin instance each for the upper and lower bound by upper and lower bounding the probability to infect a new vertex in the next step.

    Let $c \in \R_{\geq 1/2}$. Consider a state with $I=(c\infectionRate\numberOfVertices)^{-1/\infectionExponent}$ infected vertices. From this state, vertices heal at a rate of $I$ and because $(\infectionRate \numberOfVertices/2)^{-1/\infectionExponent} \leq n/2$, new vertices get infected at a rate of

\begin{align}
    \infectionRate I^{1+\infectionExponent}(\numberOfVertices-I) &\leq \infectionRate I^{1+\infectionExponent}\numberOfVertices\nonumber\\
    &= \infectionRate I(c\infectionRate\numberOfVertices)^{-\infectionExponent/\infectionExponent}\numberOfVertices\nonumber\\
    &=I/c\label{eq:1}
\end{align}

and

\begin{align}
    \infectionRate I^{1+\infectionExponent}(\numberOfVertices-I) &\geq \infectionRate I^{1+\infectionExponent}\numberOfVertices/2\nonumber\\
    &= \infectionRate I(c\infectionRate\numberOfVertices)^{-\infectionExponent/\infectionExponent}\numberOfVertices/2\nonumber\\
    &=\frac{1}{2c}I.\label{eq:2}
\end{align}

Note that these bounds are exactly the same bounds we got for negative $\infectionExponent$. However, now that $\infectionExponent$ is positive, a higher $c$ actually decreases $I$ instead of increasing it. That means that for $I$ below $(\infectionRate\numberOfVertices)^{-1/\infectionExponent}$, it is more likely to heal vertices than to infect new ones.

Let $p_1$ be the probability that the infection reaches a state with $(2\infectionRate \numberOfVertices)^{-1/\infectionExponent}$ infected vertices. As $(2\infectionRate \numberOfVertices)^{-1/\infectionExponent} \leq (\infectionRate \numberOfVertices/2)^{-1/\infectionExponent}$ it holds $p_1 \geq p$. Choosing $c=2$ in \Cref{eq:1} implies that below $(2\infectionRate \numberOfVertices)^{-1/\infectionExponent}$ infected vertices, the probability to infect a new vertex in the next step is at most half as high as the probability to heal one. Therefore the number of infected vertices in the discrete version of $\contactProcess$ dominates a gambler's ruin instance $A_1$ with a biased coin of probability $1/3$ which gives us using \Cref{pre:gamblersRuin} that

\begin{align*}
    p_1&\leq \frac{1-2^{1}}{1-2^{(2\infectionRate \numberOfVertices)^{-1/\infectionExponent}}}\\
    &= \frac{1}{2^{(2\infectionRate \numberOfVertices)^{-1/\infectionExponent}}-1}\\
    &\leq 2^{1-(2\infectionRate \numberOfVertices)^{-1/\infectionExponent}}.
\end{align*}

For the lower bound on $p$ note that the probability to infect a vertex in the next step is minimized when the number of infected vertices is 1. That corresponds to $c= (\infectionRate \numberOfVertices)^{-1}$ in \Cref{eq:2}, which gives us that as long as the infection did not die out, the probability to infect a vertex in the next step is at least $\frac{\infectionRate\numberOfVertices/2}{\infectionRate\numberOfVertices/2+1}$. Therefore the number of infected vertices of the discrete version of $\contactProcess$ is dominated by a gambler's ruin instance $A_2$ with a biased coin of probability $\frac{\infectionRate\numberOfVertices/2}{\infectionRate\numberOfVertices/2+1}$ which gives us

\begin{align*}
    p&\geq \frac{1-(\infectionRate\numberOfVertices/2)^{-1}}{1-(\infectionRate\numberOfVertices/2)^{-(\infectionRate \numberOfVertices/2)^{-1/\infectionExponent}}}\\
    &= \frac{(\infectionRate\numberOfVertices/2)^{-1}-1}{(\infectionRate\numberOfVertices/2)^{-(\infectionRate \numberOfVertices/2)^{-1/\infectionExponent}}-1}\\
    &\geq (\infectionRate\numberOfVertices/2)^{(\infectionRate \numberOfVertices/2)^{-1/\infectionExponent}}.\qedhere
\end{align*}
\end{proof}

The following lemma lower bounds the time we stay above the equilibrium. Its proof is very similar to \Cref{thm:cliqueSubLinExp} and uses that while we are a constant factor above the equilibrium, the probability to infect a new vertex is at least twice as high as the probability to heal one in the next step. That gives us an exponential survival time.

\begin{restatable}{lemma}{cliquePositiveE}\label{lem:star_positive_E}
Let $G$ be a clique with $\numberOfVertices \in \N_{>0}$ vertices. Further, let $\contactProcess$ be a contact process with infection coefficient $\infectionRate \in \R_{>0}$ and infection exponent $\infectionExponent \in \R_{>0}$ on $G$ that starts with $(\infectionRate \numberOfVertices/2)^{-1/\infectionExponent}$ infected vertices. Let $T$ be the survival time of $C$. If $(\infectionRate \numberOfVertices/2)^{-1/\infectionExponent} \in \smallO{\numberOfVertices}$, then $\E{T} \geq 2^{\bigTheta{\numberOfVertices}}$.
\end{restatable}

\begin{proof}
    We show that while there are at least $(\infectionRate \numberOfVertices/4)^{-1/\infectionExponent}$ infected vertices, the probability to infect a vertex in the next step is at least twice as high as the probability to heal one. That lets us lower bound the expected survival time by the time that a gambler's ruin instance with success probability $2/3$ takes to reach its lower end. As $(\infectionRate \numberOfVertices/2)^{-1/\infectionExponent} \in \smallO{\numberOfVertices}$, the size of that instance is in $\bigTheta{\numberOfVertices}$, which gives us the desired bound.

    Let $c \in \R_{\geq 2}$ such that $(\infectionRate\numberOfVertices/c)^{-1/\infectionExponent} \leq \numberOfVertices/2$. Consider a state with $I=(\infectionRate\numberOfVertices/c)^{-1/\infectionExponent}$ infected vertices. From this state, vertices heal at a rate of $I$ and new vertices get infected at a rate of

\begin{align}
    \infectionRate I^{1+\infectionExponent}(\numberOfVertices-I) &\geq \infectionRate I^{1+\infectionExponent}\numberOfVertices/2\nonumber\\
    &= \infectionRate I(\infectionRate\numberOfVertices/c)^{-\infectionExponent/\infectionExponent}\numberOfVertices/2\nonumber\\
    &=\frac{c}{2}I.\label{eq:3}
\end{align}

Note that while there at least $(\infectionRate\numberOfVertices/2)^{-1/\infectionExponent}$ infected vertices, new vertices get infected faster than vertices heal. That means that between $(\infectionRate\numberOfVertices/2)^{-1/\infectionExponent}$ and $\numberOfVertices/2$ infected vertices, the process dominates an unbiased gambler's ruin, which gives us a probability of at least $2/\numberOfVertices$ to reach $\numberOfVertices/2$ infected vertices before the infection dies out (see \cite{feller1957introduction}).

By choosing $c=4$ in \Cref{eq:3} we get that while there are at least $(\infectionRate\numberOfVertices/4)^{-1/\infectionExponent}$ infected vertices, new vertices get infected at least twice as fast as vertices heal. Hence, between $(\infectionRate\numberOfVertices/4)^{-1/\infectionExponent}$ and $\numberOfVertices/2$ infected vertices the number of infected vertices of the discrete version of $\contactProcess$ dominates a biased gambler's ruin instance with success probability $2/3$. Let $p$ be the probability that starting at $\numberOfVertices/2 -1$ infected vertices, the number of infected vertices drops to $(\infectionRate\numberOfVertices/4)^{-1/\infectionExponent}$ before reaching $\numberOfVertices/2$. We get

\begin{align*}
    p&\leq \frac{1-2^{1}}{1-2^{\numberOfVertices/2 - (\infectionRate\numberOfVertices/4)^{-1/\infectionExponent}}}\\
    &= \frac{1}{2^{\numberOfVertices/2 - (\infectionRate\numberOfVertices/4)^{-1/\infectionExponent}}-1}.
\end{align*}

As $(\infectionRate\numberOfVertices/2)^{-1/\infectionExponent}\in \smallO{\numberOfVertices}$ it holds $\numberOfVertices/2 - (\infectionRate\numberOfVertices/4)^{-1/\infectionExponent} \in \bigTheta{\numberOfVertices}$. After reaching $\numberOfVertices/2$ infected vertices, the process returns to a state with this many infected vertices again a number of times that dominates a geometric random variable with parameter $p$. The return times are in expectation at least $1/\numberOfVertices$ time apart as vertices heal at a rate of at most $\numberOfVertices$. Together with the fact that starting with $(\infectionRate\numberOfVertices/2)^{-1/\infectionExponent}$ infected vertices we have a probability to reach $\numberOfVertices/2$ infected vertices of at least $2/\numberOfVertices$, we get

\begin{align*}
    \E{T} &\geq \frac{2}{\numberOfVertices} \cdot \frac{1}{p} \cdot \frac{1}{\numberOfVertices}\\
    &\geq 2^{\bigTheta{\numberOfVertices}}.\qedhere
\end{align*}
\end{proof}

Putting those two results together gives us the following threshold.

\begin{corollary}
    \label{cor:cliqueSuperLinear}
Let $G$ be a clique with $\numberOfVertices \in \N_{>0}$ vertices. Further, let $\contactProcess$ be a contact process with infection coefficient $\infectionRate \in \R_{>0}$ and infection exponent $\infectionExponent \in \R_{>0}$ on $G$ that starts with exactly one infected vertex. Let $T$ be the survival time of $\contactProcess$.

\begin{enumerate}
    \item If $\infectionRate \in \smallOmega{\numberOfVertices^{-1-\infectionExponent}\log (\numberOfVertices)^{\infectionExponent}}$ then $\E{T}$ is exponential in $\numberOfVertices$.
    \item If $\infectionRate \in \smallO{\numberOfVertices^{-1-\infectionExponent}}$ then $\E{T}$ is constant in $\numberOfVertices$.\qedhere
\end{enumerate}
\end{corollary}

\begin{proof}
    First consider $\infectionRate \in \smallOmega{\numberOfVertices^{-1-\infectionExponent}\log( \numberOfVertices)^{\infectionExponent}}$. Note that this implies that $\infectionRate \numberOfVertices^{-1/\infectionExponent} \in \smallO{\numberOfVertices/ \log( \numberOfVertices)}$. Now consider all $\infectionRate$ such that $\infectionRate \numberOfVertices \geq 1$. For all larger $\infectionRate$ we get the same results using the fact that the expected survival time scales monotonically with $\infectionRate$. Now the conditions for \Cref{lem:star_positive_p} and \Cref{lem:star_positive_E} are fulfilled. Let $E_0$ be the event that the infection reaches a state with $(\infectionRate \numberOfVertices/2)^{-1/\infectionExponent}$ infected vertices. By \Cref{lem:star_positive_p} it holds that

    \begin{align*}
        \Pr{E_0} &\geq (\infectionRate \numberOfVertices/2)^{(\infectionRate \numberOfVertices/2)^{-1/\infectionExponent}}\\
        &\geq (\infectionRate \numberOfVertices/2)^{\smallO{\numberOfVertices/\log(\numberOfVertices)}}\\
        &\geq 2^{-\smallO{\numberOfVertices}}.
    \end{align*}

    By \Cref{lem:star_positive_E} it holds that $\E{T}[E_0] \geq 2^{\bigTheta{\numberOfVertices}}$. Together we get

    \begin{align*}
        \E{T} &\geq \Pr{E_0} \cdot \E{T}[E_0]\\
        &\geq 2^{-\smallO{\numberOfVertices}} \cdot 2^{\bigTheta{\numberOfVertices}}\\
        &\geq 2^{\bigTheta{\numberOfVertices}}.
    \end{align*}

    Now consider $\infectionRate \in \smallO{\numberOfVertices^{-1-\infectionExponent}}$. Let $c \in \R_{>0}$. Consider a state with $I=(c\infectionRate\numberOfVertices)^{-1/\infectionExponent}$ infected vertices. From this state, vertices heal at a rate of $I$ and new vertices get infected at a rate of

\begin{align*}
    \infectionRate I^{1+\infectionExponent}(\numberOfVertices-I) &\leq \infectionRate I^{1+\infectionExponent}\numberOfVertices\\
    &= \infectionRate I(c\infectionRate\numberOfVertices)^{-\infectionExponent/\infectionExponent}\numberOfVertices\\
    &=I/c.
\end{align*}

As $\infectionRate \in \smallO{\numberOfVertices^{-1-\infectionExponent}}$, it holds that $(\infectionRate\numberOfVertices)^{-1/\infectionExponent} \in \smallOmega{\numberOfVertices}$. Therefore for each reachable state it holds that $c \geq 2$, which means that the probability to heal a vertex is always at least twice as high as the probability to infect one. Therefore the process is dominated by a biased gambler's ruin instance with probability of $1/3$ to increase by 1. The gambler's ruin instance has a constant expected time to reach its lower bound and because triggers always happen at a rate of at least $1$, that also gives us a constant upper bound for $\E{T}$.
\end{proof}

\section{Star}
\label{sec:star}
For the star with $\numberOfVertices \in \N$ leaves we start by simplifying the process by only considering the number of infected leafs $\infectedDiscrete{\timeDiscrete}$ at step $\timeDiscrete$ and whether the center is infected or not. We then get the following transition rates

\begin{align*}
    \text{when }&\text{the center is infected:}&&\\
    &\infectedDiscrete{\timeDiscrete+1}= \infectedDiscrete{\timeDiscrete} +1 &\text{at rate } \infectionRate(\numberOfVertices - \infectedDiscrete{\timeDiscrete}),\\
    &\infectedDiscrete{\timeDiscrete+1}= \infectedDiscrete{\timeDiscrete} -1 &\text{at rate } \infectedDiscrete{\timeDiscrete},\\
    &\text{the center heals} &\text{at rate } 1.\\
    \text{when }&\text{the center is healthy:}&&\\
    &\infectedDiscrete{\timeDiscrete+1}= \infectedDiscrete{\timeDiscrete} +1 &\text{at rate } 0,\\
    &\infectedDiscrete{\timeDiscrete+1}= \infectedDiscrete{\timeDiscrete} -1 &\text{at rate } \infectedDiscrete{\timeDiscrete},\\
    &\text{the center gets infected} &\text{at rate } \infectionRate \infectedDiscrete{\timeDiscrete}^{1+\infectionExponent}.\\
\end{align*}

This time our analysis works for both positive and negative $\infectionExponent$. We first define $\infectionThreshold =\infectionRate^2\numberOfVertices(\infectionRate\numberOfVertices)^{\infectionExponent} $ because this value dictates the survival time of the process. We now start with the upper bound which is derived similarly to the normal setting.

\begin{lemma}\label{lem:starUpper}
Let $G$ be a star with $\numberOfVertices \in \N_{>0}$ leaves. Further, let $\contactProcess$ be a contact process with infection coefficient $\infectionRate \in (0,1)$ and infection exponent $\infectionExponent \in \R_{>-1}$ on $G$ that starts with infected center and susceptible leaves. Let $T$ be the survival time of $C$. If $\infectionRate \in \smallO{\numberOfVertices^{-\infectionExponent/(1+\infectionExponent)}}$ and $\infectionRate \in \smallOmega{\numberOfVertices^{-1}}$, then $\E{T} \leq 2^{\bigTheta{\infectionThreshold}}\log (\numberOfVertices)$.
\end{lemma}

\begin{proof}
    We first show that the process stays at fewer than $3\infectionRate \numberOfVertices$ infected leaves for most of the time even if the center is permanently infected. We then lower bound the probability for the infection to die out in a single center-healthy phase when it starts with at most $3 \infectionRate \numberOfVertices$ infected leaves. Bounding the expected length of center-healthy and center-infected phases concludes the proof.

    Consider the modified process $C'$ that behaves like $\contactProcess$ with the exception that it ignores all of the healing triggers of the center. The number of infected leaves in $C'$ dominates the number of infected leaves in $C$. While there are $I'$ infected leaves in $C'$, leaves heal at a rate of $I'$ and new leaves get infected at a rate of $\infectionRate(\numberOfVertices - I')\leq \infectionRate \numberOfVertices$. Therefore, while $I' \geq 2\infectionRate \numberOfVertices$, leaves heal at least twice as fast as new leaves get infected. That means that between $2\infectionRate\numberOfVertices$ and $3\infectionRate\numberOfVertices$ infected leaves, the discrete version of the process is dominated by a gambler's ruin instance with a biased coin with success probability $1/3$. As this instance takes in expectation an exponential time in $\infectionRate\numberOfVertices$ to reach its upper limit and only a linear time to drop back to its lower bound, the time that $C'$ spends above $3\infectionRate\numberOfVertices$ infected leaves is negligible compared to the entire survival time $T$. As $C'$ dominates $C$, the same holds for $C$. Therefore, in the following we assume that the number of infected leaves stays below $3\infectionRate \numberOfVertices$ without reducing the expected survival time too much with this assumption.

    Consider a state with at most $3\infectionRate \numberOfVertices$ infected leaves and healthy center. Let $E$ be the event that the infection dies out before infecting the center. In a state with $I$ infected leaves, leaves heal at a rate of $I$ and the center gets infected at rate $\infectionRate I^{1+\infectionExponent}$. In order for $E$ to happen, all of the leaves have to heal before the center gets infected. It holds for all $x \in [0,1.59]$ that $e^{-x} \leq 1- x/2$. Note that $\infectionRate \in \smallO{\numberOfVertices^{-\infectionExponent/(1+\infectionExponent)}}$ implies $\infectionRate(\infectionRate \numberOfVertices)^{ \infectionExponent} \in \smallO{1}$. Hence, for sufficiently large $\numberOfVertices$, we can use the previous inequality with $x/2= \infectionRate(\infectionRate \numberOfVertices)^{ \infectionExponent}$. With that, we bound the probability of $E$ by

    \begin{align}
    \Pr{E} &\geq \prod_{i=1}^{3 \infectionRate \numberOfVertices}{\frac{i}{i+\infectionRate i^{1+\infectionExponent}}}\nonumber = \prod_{i=1}^{3 \infectionRate \numberOfVertices}{\frac{1}{1+\infectionRate i^{\infectionExponent}}}\nonumber\\
    &= \prod_{i=1}^{3 \infectionRate \numberOfVertices}{\left(1-\frac{\infectionRate i^{\infectionExponent}}{1+\infectionRate i^{\infectionExponent}}\right)}\nonumber \geq \prod_{i=1}^{3 \infectionRate \numberOfVertices}{\left(1-\infectionRate i^{\infectionExponent}\right)}\nonumber\\
    &\geq \prod_{i=1}^{3 \infectionRate \numberOfVertices}{e^{-2 \infectionRate i^{\infectionExponent}}}\nonumber = e^{-2 \infectionRate \sum_{i=1}^{3 \infectionRate \numberOfVertices}{i^{\infectionExponent}}}\nonumber\\
    &= e^{-\bigTheta{\infectionRate (\infectionRate \numberOfVertices)^{1+\infectionExponent}}}\nonumber \\
    &= e^{-\bigTheta{\infectionThreshold}}.\label{eq:4}
\end{align}

The second to last step uses the Euler-Maclaurin Formula (see \cite{kac2002quantum}).

Let $S$ be the number of center-healthy phases of $\contactProcess$ that start with at most $3\infectionRate\numberOfVertices$ infected leaves before the infection dies out. By \Cref{eq:4}, $S$ is dominated by a geometric random variable with parameter $e^{-\bigTheta{\infectionThreshold}}$. Each center-infected phase lasts in expectation for 1 time unit as it ends when the center heals which happens at rate 1. By \Cref{pre:maxExponential}, each center-healthy phase lasts at most $\log( \numberOfVertices)$ time units in expectation as it ends the latest when the last leaf heals which is determined by the maximum of $\numberOfVertices$ exponential random variables with mean 1. To bound $T$, we also need to consider the time $T'$ spent above $3\infectionRate\numberOfVertices$ infected leaves, but as we argued before, that time is much smaller than the rest of $T$. We get

\begin{align*}
    \E{T} &\leq \E{S} \cdot (1+ \log (\numberOfVertices)) + T'\\
    &\leq e^{\bigTheta{\infectionThreshold}} \cdot (1+ \log (\numberOfVertices)) + T'\\
    &\leq 2^{\bigTheta{\infectionThreshold}}\log( \numberOfVertices).\qedhere
\end{align*}
\end{proof}

For the lower bound, the idea is to look at the process while the number of infected vertices is in between $\infectionRate \numberOfVertices /8$ and $\infectionRate \numberOfVertices/4$ and splitting this range into $\sqrt{\infectionThreshold}$ many equally sized blocks. We then use the center-healthy phases and center-infected phases as the steps of our process. Now in a center-healthy phase, the probability of decreasing the number of infected vertices by more than a block is exponentially small in $\sqrt{\infectionThreshold}$. In a center-healthy phase, the probability of not increasing by a block before healing the center is sub constant. We now build a gambler's ruin game on the $\sqrt{\infectionThreshold}$ many blocks as states by taking as a step a full cycle between two center infections. The probability of decreasing by more than one block is so small that we ignore it. The probability of increasing by a block in total is much higher than the probability of decreasing by one (much more than twice as much). So our process dominates a biased gambler's ruin on $\sqrt{\infectionThreshold}$ many states with a biased coin of probability $2/3$. Therefore, the time it takes to die out is with high probability exponential in $\sqrt{\infectionThreshold}$.

We first bound the probability of healing too many vertices in a center-healthy phase.

\begin{lemma}\label{lem:starHealthy}
Let $G$ be a star with $\numberOfVertices \in \N_{>0}$ leaves. Further, let $\contactProcess$ be a contact process with infection coefficient $\infectionRate \in \R_{>0}$ and infection exponent $\infectionExponent \in \R_{>-1}$ on $G$. Let $x,y \in \N$. Further, let $p^x_y$ be the probability to drop in a center-healthy phase from $x$ infected vertices to at most $y$. Let $z=y$ if $\infectionExponent$ is positive and $z=x$ otherwise. If $\infectionRate z^{\infectionExponent} \leq 1$, then $p^x_y \leq e^{-(x-y)\frac{\infectionRate z^{\infectionExponent}}{2}}$.
\end{lemma}

\begin{proof}
    While there are $I \in \N$ infected leaves, leaves heal at a rate of $I$ and the center gets infected at a rate of $\infectionRate I^{1+\infectionExponent}$. Hence, in the next step a leaf heals with a probability of $\frac{I}{I+\infectionRate I^{1+\infectionExponent}} = \frac{1}{1+\infectionRate I^{\infectionExponent}}$. Note that this probability is monotonically decreasing or increasing in $I$ when $\infectionExponent$ is positive or negative respectively. For our choice of $z$, that means that in the interval between $y$ and $x$, this probability is maximized at $z$. Now in order to drop from $x$ infected leaves to $y$ in a center-healthy phase, in each state in between a leaf has to heal. Using $\infectionRate z^{\infectionExponent} \leq 1$, we bound that probability by

    \begin{align*}
        p^x_y &= \prod_{i=y+1}^{x}{\frac{1}{1+\infectionRate i^{\infectionExponent}}}\\
    &\leq \left(\frac{1}{1+\infectionRate z^{\infectionExponent}}\right)^{x-y}\\
    &= \left(1-\frac{\infectionRate z^{\infectionExponent}}{1+\infectionRate z^{\infectionExponent}}\right)^{x-y}\\
    &\leq \left(1-\frac{\infectionRate z^{\infectionExponent}}{2}\right)^{x-y}\\
    &\leq e^{-(x-y)\frac{\infectionRate z^{\infectionExponent}}{2}}.\qedhere
    \end{align*}
\end{proof}

Now we lower bound the probability of infecting enough vertices in a center-infected phase.

\begin{lemma}\label{lem:starInfected}
Let $G$ be a star with $\numberOfVertices \in \N_{>0}$ leaves and let $z\in \R_{>0}$. Further, let $\contactProcess$ be a contact process with infection coefficient $\infectionRate \in (0,1)$ and infection exponent $\infectionExponent \in \R_{>-1}$ on $G$ that starts with infected center and at most $\infectionRate \numberOfVertices/4-\frac{\infectionRate \numberOfVertices}{4 z}$ infected leaves. Let $\infectionRate\numberOfVertices\in \smallOmega{1}$. Then if $z \in \smallO{\infectionRate\numberOfVertices}$, the probability of the event $E$ that the number of infected leaves increases by at least $\frac{\infectionRate \numberOfVertices}{4 z}$ before the center heals is at least $e^{-\bigTheta{1/z}}$.
\end{lemma}

\begin{proof}
    We first show that it is likely that there are at least $\frac{\infectionRate \numberOfVertices}{ z}$ steps that do not heal the center and to then lower bound the probability that enough of these steps infect new leaves.

    While the center is infected, each leaf either heals at a rate of 1 or gets infected at a rate of $\infectionRate$ depending on whether it is infected or not. Thus, each leaf changes its state at a rate of at least $\infectionRate$. The center heals at rate one. That means that each step has a probability of at most $\frac{1}{\infectionRate\numberOfVertices}$ to heal the center. Therefore, the number of steps $S$ before the center heals dominates a geometric random variable $X$ with parameter $\frac{1}{\infectionRate\numberOfVertices}$. We get

    \begin{align*}
    \Pr{S \geq \frac{\infectionRate \numberOfVertices}{z}}
    &\geq \Pr{X \geq \frac{\infectionRate \numberOfVertices}{z}}\\
    &\geq \left(1-\frac{1}{\infectionRate \numberOfVertices}\right)^{ \frac{\infectionRate \numberOfVertices}{ z}}\\
    &\geq e^{- \frac{1}{\infectionRate \numberOfVertices} \frac{\infectionRate \numberOfVertices}{ z}}\\
    &\geq e^{-\bigTheta{\frac{1}{z}}}.
    \end{align*}

Now when there are at least $\frac{\infectionRate \numberOfVertices}{z}$ steps before the center heals, then if at least $5/8$-th of those infect new vertices that implies $E$. While there are at most $\infectionRate\numberOfVertices/4$ infected leaves, leaves heal at a rate of at most $\infectionRate\numberOfVertices/4$ and new vertices get infected at a rate of at least $\infectionRate(\numberOfVertices-\infectionRate\numberOfVertices/4)\geq \infectionRate\numberOfVertices/2$. Hence, each step has a probability of at least $2/3$ to infect a leaf. That means that the number of the steps out of the first $S$ steps that infect a new leaf dominates a binomial random variable $B \sim \text{Bin}(S,2/3)$. We get using Chernoff bounds (\Cref{pre:chernoff})

\begin{align*}
    \Pr{E}[S \geq \frac{\infectionRate \numberOfVertices}{z}] &\geq \Pr{B \geq \frac{5}{8}S}[S \geq \frac{\infectionRate \numberOfVertices}{z}]\\
    &\geq \Pr{B \geq (1-\frac{1}{16})\E{B}}[S \geq \frac{\infectionRate \numberOfVertices}{z}]\\
    &\geq 1- e^{-\frac{1}{2 \cdot 16^2}\cdot \frac{\infectionRate\numberOfVertices}{z}}.
\end{align*}

This gives us a lower bound for the probability of $E$ of

\begin{align*}
    \Pr{E} &\geq \Pr{S \geq \frac{\infectionRate \numberOfVertices}{z}} \cdot \Pr{E}[S \geq \frac{\infectionRate \numberOfVertices}{z}]\\
    &\geq e^{-\bigTheta{\frac{1}{z}}} \cdot (1- e^{-\frac{1}{2 \cdot 16^2}\cdot \frac{\infectionRate\numberOfVertices}{z}})\\
    &\geq e^{-\bigTheta{\frac{1}{z}}}.\qedhere
\end{align*}
\end{proof}

Putting those results together gives us the following lower bound.

\begin{theorem}\label{thm:starLowerTime}
Let $G$ be a star with $\numberOfVertices \in \N_{>0}$ leaves. Further, let $\contactProcess$ be a contact process with infection coefficient $\infectionRate \in (0,1)$ and infection exponent $\infectionExponent \in \R_{>-1}$ on $G$ that starts with infected center and at least $\infectionRate \numberOfVertices/4$ infected leaves. Let $T$ be the survival time of $C$. If $\infectionRate (\infectionRate \numberOfVertices)^{\infectionExponent} \leq 1$ and $\infectionThreshold \in \smallOmega{1}$, then a.a.s. it holds $T \geq 2^{\bigTheta{\sqrt{\infectionThreshold}}}$.
\end{theorem}

\begin{proof}
    We look at the process while the number of infected vertices is in between $\infectionRate \numberOfVertices /8$ and $\infectionRate \numberOfVertices/4$ and split this interval into $\sqrt{\infectionThreshold}$ many equally sized blocks. We then only consider in which of those blocks the number of infected vertices is each time the center gets infected. We show that the resulting process with high probability dominates a gambler's ruin instance with a biased coin of probability $2/3$ which gives us the desired bound.

    We define the process $X$ that is coupled to $C$ as follows. $X$ transitions to different values at exactly the times at which the center gets infected in $C$. It takes as values the number of infected leaves in $C$ divided by $\frac{\infectionRate\numberOfVertices}{8\sqrt{\infectionThreshold}}$ rounded down to an integer. To bound the survival time $T$, we consider the process until one of two events happens: Either $X$ decreases by more than 1 in a step or it reaches $\sqrt{\infectionThreshold}$. We only consider $C$ below $\infectionRate\numberOfVertices/4$ infected vertices. Every time when a center gets infected while more leaves are infected we ignore it.

    In order for $X$ to reduce by more than 1 in a step, either the center-infected phase or the center-healthy phase has to reduce the number of infected leaves by at least $\frac{\infectionRate\numberOfVertices}{16\sqrt{\infectionThreshold}}$. While below $\infectionRate\numberOfVertices/4$ infected leaves, center-infected phases heal vertices at least twice as fast as they heal leaves which makes it exponentially unlikely to reduce the number of infected leaves by too much. For the center-healthy phases, by \Cref{lem:starHealthy} the probability to reduce the number of infected leaves in a center-healthy phase by $\frac{\infectionRate\numberOfVertices}{16\sqrt{\infectionThreshold}}$ while the number of infected leaves is between $\infectionRate\numberOfVertices/8$ and $\infectionRate\numberOfVertices/4$ is at most $e^{-\bigTheta{\frac{\infectionRate\numberOfVertices}{\sqrt{\infectionThreshold}}\cdot \infectionRate (\infectionRate\numberOfVertices)^\infectionExponent}}= e^{-\bigTheta{\sqrt{\infectionThreshold}}}$. That means that the time until that event happens is at least geometrically distributed with probability $e^{-\bigTheta{\sqrt{\infectionThreshold}}}$ which means that it takes a.a.s. $e^{\bigTheta{\sqrt{\infectionThreshold}}}$ until then.
    
    Now assume that $X$ never reduces by more than 1 in a step. By \Cref{lem:starInfected}, the probability that $X$ increases by at least 1 in a step is at least $e^{-\bigTheta{1/\sqrt{\infectionThreshold}}}$ which is at least $2/3$ for sufficiently large $\numberOfVertices$ as $\infectionThreshold \in \smallOmega{1}$. That means that $X$ dominates a gambler's ruin instance with a biased coin of probability $2/3$ in the range between $\sqrt{\infectionThreshold}$ and $2\sqrt{\infectionThreshold}$. This gambler's ruin instance has a.a.s. a time of $2^{\bigTheta{\sqrt{\infectionThreshold}}}$ until it reaches its lower bound.

    Now the infection cannot die out before $X$ reaches $\sqrt{\infectionThreshold}$. We argued that this does not happen before either $X$ reduces by more than 1 in a step or the gambler's ruin instance reaches $\sqrt{\infectionThreshold}$. As shown before, both of those events take a.a.s. $2^{\bigTheta{\sqrt{\infectionThreshold}}}$ time. As each phase considered in $X$ needs the center to heal and then infect which takes at least 1 time unit in expectation, this gives us the desired bound for $T$.
\end{proof}

We now show that starting with only the center infected, we reach a state with $\infectionRate\numberOfVertices/4$ infected leaves with high probability.

\begin{lemma}\label{lem:starLowerReach}
    Let $G$ be a star with $\numberOfVertices \in \N_{>0}$ leaves. Further, let $\contactProcess$ be a contact process with infection coefficient $\infectionRate \in (0,1)$ and infection exponent $\infectionExponent \in \R_{>-1}$ on $G$ that starts with infected center and susceptible leaves. If $\infectionRate (\infectionRate \numberOfVertices)^{\infectionExponent} \leq 1$ and $\infectionThreshold \in \smallOmega{\log(n)^4}$, $C$ reaches a state with $\infectionRate\numberOfVertices/4$ infected leaves asymptotically almost surely.
\end{lemma}

\begin{proof}
    We show that asymptotically almost surely we reach a state from which we need a super-constant amount of center-healthy phases for the infection to die out. We then show that each center-infected phase has a constant probability to reach $\infectionRate\numberOfVertices/4$ infected leaves. That makes it very unlikely not to reach $\infectionRate\numberOfVertices/4$ infected leaves before the infection dies out.

    By \Cref{lem:starInfected}, the probability that the process reaches a state with at least $\frac{\infectionRate\numberOfVertices}{ \infectionThreshold^{\frac{1}{2(1+\infectionExponent)}}}$ infected leaves is at least $e^{-\bigTheta{1/\infectionThreshold^{\frac{1}{2(1+\infectionExponent)}}}}$. Now consider phases of the process that start and end with the center infecting or the infection dying out. By \Cref{lem:starHealthy} between $\frac{\infectionRate\numberOfVertices}{2 \infectionThreshold^{\frac{1}{2(1+\infectionExponent)}}}$ and $\frac{\infectionRate\numberOfVertices}{ \infectionThreshold^{\frac{1}{2(1+\infectionExponent)}}}$ infected leaves, the probability $p$ to heal more than $\frac{\infectionRate\numberOfVertices}{ \infectionThreshold^{\frac{1}{2(1+\infectionExponent)}}\cdot \sqrt[4]{\infectionThreshold}}$ is at most

    \begin{align*}
        p &\leq e^{-\frac{\infectionRate\numberOfVertices}{ \infectionThreshold^{\frac{1}{2(1+\infectionExponent)}}\cdot \sqrt[4]{\infectionThreshold}} \cdot \infectionRate (\frac{\infectionRate\numberOfVertices}{2 \infectionThreshold^{\frac{1}{(1+\infectionExponent)}}})^\infectionExponent}\\
        &\leq e^{-\frac{\infectionThreshold}{\infectionThreshold^{\frac{1+\infectionExponent}{2(1+\infectionExponent)}}\sqrt[4]{\infectionThreshold}}}\\
        &\leq e^{-\sqrt[4]{\infectionThreshold}}.
    \end{align*}

    As $\infectionThreshold \in \smallOmega{\log(n)^4}$, that means that asymptotically almost surely, there are at least $\sqrt[4]{\infectionThreshold}$ many of those phases before the infection dies out. By \Cref{lem:starInfected}, each of those phases has a constant probability to infect $\infectionRate\numberOfVertices/4$ many leaves. Therefore we reach a state with $\infectionRate\numberOfVertices/4$ infected leaves before the infection dies out asymptotically almost surely.
\end{proof}

Plugging in the values for $\infectionRate$ and combining the previous lemmas gives us the following bounds

\begin{corollary}
    \label{cor:star}
Let $G$ be a star with $\numberOfVertices \in \N_{>0}$ leaves. Further, let $\contactProcess$ be a contact process with infection coefficient $\infectionRate \in \R_{>0}$ and infection exponent $\infectionExponent \in \R_{>-1}$ on $G$ that starts with infected center and susceptible leaves. Let $T$ be the survival time of $\contactProcess$.

\begin{enumerate}
    \item If $\infectionRate \in \smallOmega{\numberOfVertices^{-1/2-\frac{\infectionExponent}{2(2+\infectionExponent)}}\log (\numberOfVertices)^{4/(2+\infectionExponent)}}$ then $T$ is a.a.s. super-polynomial in $\numberOfVertices$.
    \item If $\infectionRate \in \bigO{\numberOfVertices^{-1/2-\frac{\infectionExponent}{2(2+\infectionExponent)}}}$ then $\E{T}$ is at most logarithmic in $\numberOfVertices$.\qedhere
\end{enumerate}
\end{corollary}

\begin{proof}
    First consider $\infectionRate \in \smallOmega{\numberOfVertices^{-1/2-\frac{\infectionExponent}{2(2+\infectionExponent)}}\log( \numberOfVertices)^{4/(2+\infectionExponent)}}$. We show the survival time lower bound for $\infectionRate \in \smallO{\numberOfVertices^{-\infectionExponent/(1+\infectionExponent)}}$ and $\infectionRate \in \smallO{1}$ as it also holds for all larger $\infectionRate$ then because of the monotonicity of the survival time. Note that such a $\infectionRate$ always exists as $-1/2-\frac{\infectionExponent}{2(2+\infectionExponent)}<-\infectionExponent/(1+\infectionExponent)$ and $-1/2-\frac{\infectionExponent}{2(2+\infectionExponent)}<0$ for all $\infectionExponent \in \R \setminus [-2,-1]$. Now $\infectionRate \in \smallO{\numberOfVertices^{-\infectionExponent/(1+\infectionExponent)}}$ and $\infectionRate \in \smallO{1}$ imply both $\infectionRate \leq 1$ and $\infectionRate(\infectionRate\numberOfVertices)^{\infectionExponent}\leq 1$ for sufficiently large $\numberOfVertices$. The fact that $\infectionRate \in \smallOmega{\numberOfVertices^{-1/2-\frac{\infectionExponent}{2(2+\infectionExponent)}}\log( \numberOfVertices)^{4/(2+\infectionExponent)}}$ implies that $\infectionThreshold \in \smallOmega{\log(\numberOfVertices)^4}$. Hence, both \Cref{thm:starLowerTime} and \Cref{lem:starLowerReach} are applicable.

    By \Cref{lem:starLowerReach}, $\contactProcess$ reaches a state with $\infectionRate\numberOfVertices/4$ a.a.s and by \Cref{thm:starLowerTime} it then survives a.a.s. for a time of at least $2^{\bigTheta{\sqrt{\infectionThreshold}}}$. As $\infectionThreshold \in \smallOmega{\log(\numberOfVertices)^4}$, this time is super-polynomial in $\numberOfVertices$ which concludes the proof for the first case.

    Now consider $\infectionRate \in \bigO{\numberOfVertices^{-1/2-\frac{\infectionExponent}{2(2+\infectionExponent)}}}$. We show the survival time upper bound for $\infectionRate \in \bigTheta{\numberOfVertices^{-1/2-\frac{\infectionExponent}{2(2+\infectionExponent)}}}$ and it follows for smaller $\infectionRate$ because of the linearity of the survival time. Note that for $\infectionExponent > -2$ it holds $-1/2-\frac{\infectionExponent}{2(2+\infectionExponent)}>-1$ and hence $\infectionRate \in \smallOmega{\numberOfVertices^{-1}}$. Also for all $\infectionExponent \in \R \setminus [-2,-1]$ it holds $-1/2-\frac{\infectionExponent}{2(2+\infectionExponent)}<-\infectionExponent/(1+\infectionExponent)$ and $-1/2-\frac{\infectionExponent}{2(2+\infectionExponent)}<0$ which imply $\infectionRate \in \smallO{\numberOfVertices^{-\infectionExponent/(1+\infectionExponent)}}$ and $\infectionRate < 1$ for sufficiently large $\numberOfVertices$. Hence \Cref{lem:starUpper} is applicable which gives us that $\E{T} \leq 2^{\bigTheta{\infectionThreshold}}\log( \numberOfVertices)$. Noting that for $\infectionRate \in \bigTheta{\numberOfVertices^{-1/2-\frac{\infectionExponent}{2(2+\infectionExponent)}}}$, $\infectionThreshold \in \bigTheta{1}$, concludes the proof.
\end{proof}

\section*{Acknowledgments}

The authors thank George Skretas for helpful discussions and his guidance on the literature of influence maximization.

Andreas Göbel was funded by the project PAGES (project No. 467516565) of the German Research Foundation (DFG).

Marcus Pappik was funded by the HPI Research School on Data Science and Engineering.

\printbibliography

\end{document}